\documentclass[12pt, twoside, leqno]{article}
\usepackage{authblk}
\usepackage{amssymb}
\usepackage{amsmath}
\usepackage{indentfirst}
\usepackage{t1enc}
\usepackage{array}
\usepackage{graphpap}
\usepackage[mathscr]{eucal}
\usepackage{enumerate}
\usepackage{tabularx}
\usepackage{graphicx}
\usepackage[colorlinks=true,linkcolor=blue,citecolor=blue,
  urlcolor=blue]{hyperref}
\usepackage[hyphenbreaks]{breakurl}
\usepackage{amsthm}
\usepackage{graphicx}
\usepackage{float} 

\usepackage[normalem]{ulem}

\usepackage{tikz}
\usepackage{pgfplots}
\usetikzlibrary{shapes,arrows}

\usepackage[latin2,utf8]{inputenc}

\theoremstyle{plain}
\newtheorem{thm}{Theorem}
\newtheorem{cor}{Corollary}
\newtheorem{lem}{Lemma}

\newtheorem{conj}{Conjecture}

\newcommand{\mc}{\mathcal}

\newcommand{\ab}[1]{\left\vert{#1}\right\vert}
\newcommand{\zj}[1]{\left({#1}\right)}

\newcommand{\lb}[1]{\label{#1}}
\allowdisplaybreaks[1]

\normalsize

\title{Erdős--Turán Theorem and Eulerian Integers}

\author[a]{Erik Füredi}
\author[b]{Katalin Gyarmati}
\affil[a]{ELTE Eötvös Loránd University Faculty of Science, Budapest 1117, Hungary, email: erikfuredi@gmail.com}
\affil[b]{ELTE Eötvös Loránd University, Department of Algebra and Number Theory, Budapest 1117, Hungary, email: katalin.gyarmati@gmail.com}
\date{2025. 10. 12.}

\begin{document}
\baselineskip=17pt
\maketitle

\footnotetext{
\indent Research supported by the Hungarian National Research Development and
Innovation Funds KKP133819.}
\begin{abstract}
Our work is motivated by the fact that the norms of the Eulerian integers
are related to the sums of form $a^2-ab+b^2$, providing a natural generalization for problems concerning products over sums or differences of integers.
Let $E$ be the set of Eulerian integers. We define $\omega_{\mathbb N}(x)$
as the number of distinct prime divisors of $x\in\mathbb N$, and $\omega_E(x)$
as the number of distinct Euler prime divisors of $x\in E$.
By the Erdős--Turán theorem, if $\mc A\subset\mathbb Z^{+}$ and $|\mathcal{A}|=3\cdot{2^{k-1}}$ ($k\in\mathbb{Z}^+$), then $\omega_\mathbb{N}(\prod_{a,b\in\mathcal{A},a\neq{b}}(a+b))\geq{k+1}$.
We prove that if $\mathcal{A} \subset E$ is a finite set and $\rho \in E$, then the value of $\omega_E(\prod_{a,b \in \mathcal{A}, a \neq b}(a+\rho b))$ has a lower bound of order $\log|\mathcal{A}|$.
Consequently, we provide lower bounds for $\mathcal{A} \subset \mathbb{N}$ for both $\omega_{\mathbb{N}}(\prod_{a,b \in \mathcal{A}, a \neq b}(a^2+ab+b^2))$ and $\omega_{\mathbb{N}}(\prod_{a,b \in \mathcal{A}, a \neq b}(a^2-ab+b^2))$.
We also give an upper bound for the minimum of $\omega_{\mathbb{N}}(\prod_{a,b \in \mathcal{A}, a \neq b}(a^2+ab+b^2))$
with a computer program, if $|\mathcal{A}|\le 8$ and sets whose largest element is relatively small.
Furthermore, using a Diophantine number theoretical lemma of Győry, Sárközy, and Stewart, we
give a lower bound of order $\log|\mathcal{A}|$ for $\omega_{\mathbb{N}}(\prod_{a \in \mathcal{A}, b \in \mathcal{B}}(f(a,b)))$ for a specific class of polynomials $f \in \mathbb{Z}[x,y]$ and finite sets $\mathcal{A}, \mathcal{B} \subset \mathbb{Z}$.

\end{abstract}

\section{Introduction}

In 1934, Erd\H{o}s and Tur\'an proved their
celebrated theorem, stated below as Theorem A \cite{ErdosTuran}:

\bigskip\noindent\textbf{Theorem A [Erdős--Turán].}
\textit{If $k\in\mathbb N$, $\mathcal{A}\subseteq\mathbb{N}$ and $|\mc A|\ge 3\cdot 2^{k-1}$
is finite set then}
\[
\omega_{\mathbb N} (\prod_{\substack{a,b\in\mc A\\ a\ne b}}
(a+b))\ge k+1,
\]
\textit{where, for $n\in\mathbb N$, $\omega_{\mathbb N} (n)$ denotes the number of distinct positive prime factors of
$n$.}

\bigskip\noindent Later, an improved version of this result was presented in the Erdős--Surányi book \cite{ErdosSuranyi}, where the lower bound was proved for $|\mathcal{A}| \ge 2^k+1$.
In 1986, Győry, Stewart, and Tijdeman \cite{GyoryStewartTijdeman}
generalized the theorem to two different sets, proving
the following result:

\bigskip\noindent\textbf{Theorem B.}
\textit{There exists an effectively computable positive constant $c$ such that if
$\mc A,\ \mc B\subseteq\mathbb Z^+$ are finite sets and
$|\mc{A}|\geq|\mc{B}|\geq 2$, then}
\[
\omega_{\mathbb N}(\prod_{a\in{\mc A},\ b\in{\mc B}}(a+b))\geq{c\log|\mc A|}.
\]

In 1988, Erdős, Stewart, and Tijdeman \cite{ErdosStewartTijdeman}
proved that the lower bound in Theorem B cannot be improved significantly:

\bigskip\noindent\textbf{Theorem C.}
\textit{For any $\varepsilon>0$ and all sufficiently large
integers $k$, there exist sets $\mc A$ and $\mc B$ such that}
$|\mc A|=k$, $|\mc B|=2$, and
\[
\omega_{\mathbb N}(\prod_{a\in{\mc A},\ b\in{\mc B}}(a+b))
<\zj{\dfrac{1}{8}+\varepsilon} (\log |\mc A|)^2\log\log |\mc A|.
\]

Inspired by these previous results, we studied whether the Erd\H{o}s--Tur\'an Theorem holds in the ring of Eulerian integers.
For this, we will use the following notations: let $\omega=\dfrac{-1+i\sqrt{3}}{2}$,
a third root of unity, and let $E$ denote the ring of Eulerian integers,
i.e.,
$E=\{a+b\omega :\ a,b\in\mathbb Z\}$.
Furthermore, for $x\in E$, $\omega_E(x)$ denotes the number of distinct prime divisors of the Eulerian integer
$x$ (in the above notation, two primes are considered distinct if they are not associated). Then the analogue of
Theorem A among Eulerian integers is as follows:

\begin{thm}\lb{thm01}
Let $\mc A\subseteq E$ be a finite set such that $|\mathcal{A}|\geq2$. Then
\[
\omega_E(\prod_{\substack{a,b\in{\mc A}\\ a\ne b}}(a+b))
> \dfrac{\log (|\mc A|-1)-\log 18}{\log 2}.
\]
\end{thm}

Our proof utilizes the Law of Cosines alongside the method of Erdős and Turán.

\bigskip
It is natural to ask whether a similar statement holds for the product of sums of the type
$a+\rho b$, where $\rho$ is an arbitrary
Eulerian integer. We proved the following:

\begin{thm}\lb{thm02}
Let $\rho\in E$. Then, there exists a constant
$c$ that depends only on $\rho$ such that for any finite set $\mc A\subseteq E$
\[
\omega_E(\prod_{\substack{a,b\in{\mc A}\\ a\ne b}}(a+\rho b))
>\dfrac{\log|\mc A|}{\log 3}-c.
\]
\end{thm}

The above theorem, combined with the multiplicative property of
the norm $N(a+b\omega)=a^2-ab+b^2$
in the ring of Eulerian integers, yields surprising corollaries for
rational integers.

\begin{cor}\lb{cor01}
If $\mathcal{A}\subseteq\mathbb{Z}^+$ is a finite nonempty set, then \[\omega_{\mathbb N}(\prod_{\substack{a,b\in{\mathcal{A}}\\a\neq{b}}} (a^2-ab+b^2))>\frac{\log|\mc{A}|-\log 38}{2\log 3}.
\]
\end{cor}

\begin{cor}\lb{cor02}

If $\mathcal{A}\subseteq\mathbb{Z}^+$ is a finite nonempty set, then \[\omega_{\mathbb N}(\prod_{\substack{a,b\in{\mathcal{A}}\\a\neq{b}}} (a^2+ab+b^2))>\frac{\log|\mc{A}|-\log 146}{2\log 3}.
\]
\end{cor}

These two corollaries motivated us to ask
whether the Erdős-Turán theorem
can be generalized to arbitrary two-variable polynomials and two different sets. We conjecture the following:

\begin{conj}\lb{conj01}
If $f\in\mathbb Z[x,y]$ is a two-variable polynomial that is not decomposable as $f(x,y)=g(x)h(y)$,
where $g,h\in\mathbb Z[x]$, then there exists a constant $c$ that depends only on the polynomial $f$, such that if
$\mc A,\ \mc B\subseteq\mathbb Z^+$ are finite sets and
$|\mc{A}|\geq|\mc{B}|\geq 2$, then
\[
\omega_{\mathbb N}(\prod_{a\in{\mc A},b\in{\mc B}}(f(a,b)))\geq{c\log|\mc A|}.
\]
\end{conj}

However, while we could not prove this conjecture in its full generality, we were able to handle an important special case:

\begin{thm}\lb{thm03}
Let $n\geq2$ be a positive integer, $f(x,y)=\sum_{i=1}^{n-1} r_{i} x^{m_i}y^{i-1}+r_ny^{n-1}\in\mathbb Z[x,y]$
be a polynomial where
$m_1$,$m_2$,\dots,$m_{n-1}$ are nonnegative integers and the coefficients $r_{i}$ are positive. Then there exists an effectively computable positive constant $c$ such that if
$\mc A$ and $\mc B\subseteq\mathbb Z^{+}$ are finite sets, where $|\mc A|\ge|\mc B|\ge 2n-2$, then
\[
\omega_{\mathbb N}\Big(\prod_{a\in\mc A,
b\in\mc B}(f(a,b)))\Big)
\geq{c\log|\mc{A}|}.
\]
\end{thm}

The proof of Theorem \ref{thm03} 
relies on a theorem of
Győry, Sárközy and Stewart \cite{GyoryStewartSarkozy} and utilizes \mbox{Vandermonde} determinants.
Interestingly, this general theorem implies a lower bound of logarithmic magnitude, similar to Corollary \ref{cor02} (but not to Corollary \ref{cor01}).

If $f(a,b)=a^2+ab+b^2$, we used a Python program to examine the smallest possible values of $\omega_{\mathbb{N}}\left(\prod_{\substack{a,b\in\mathcal{A}, a\neq{b}}}(a^2+ab+b^2)\right)$ for sets of positive integers $\mathcal{A}$ with $3-8$ elements. Our analysis focused on sets whose
elements were bounded by a few hundred, specifically considering primitive sets (where the greatest common divisor of all elements is $1$).
The Table \ref{Table1} shows our such results.

\begin{table}[H]
\begin{tabular}{ |m{2cm} | m{2.5cm} | m{3cm} | m{6cm} |}
\hline
Size of $\mathcal{A}$ & Maximal allowed element & Minimum number of different prime divisors & Examples \\\hline
3 & 400 & 3 & 28868 pcs, e.g. \{1,2,3\}, \{1,2,4\}, \{388,395,399\} \\\hline
4 & 400 & 4 & 5 pcs: \{1,2,4,8\}, \{1,3,9,18\}, \{1,3,9,27\}, \{1,4,16,22\}, \{1,9,15,18\} \\\hline
5 & 200 & 5 & 2 pcs: \{1,2,4,8,16\}, \{1,3,9,27,81\} \\\hline
6 & 200 & 6 & 1 pc: \{1,2,4,8,16,32\} \\\hline
7 & 150 & 7 & 1 pc: \{1,2,4,8,16,32,64\}\\\hline
8 & 100 & 9 & 3 pcs, e.g.: \{2,3,4,6,9,12,18,36\} \\\hline
\end{tabular}
\caption{The smallest possible values of $\omega_{\mathbb{N}}\left(\protect\prod_{\protect\substack{a,b\in\mathcal{A}, a\neq{b}}}(a^2+ab+b^2)\right)$ for special sets $\mathcal{A}$ with $3-8$ elements}
\label{Table1}
\end{table}

As a continuation of this work, we plan to generalize Theorem \ref{thm03} to an arbitrary homogeneous polynomial $f(x,y)$, focusing initially on the single-set problem where we seek the bound
\[
\omega_{\mathbb N}\left(\prod_{a,b\in\mathcal{A}} (f(a,b))\right)
\geq c\log|\mathcal{A}|.
\]
However, the full proof will involve significant further complications
and will be presented in a subsequent paper.

\section{Proofs}

\bigskip\noindent\textbf{Proof of Theorem \ref{thm01}.}
Let us plot the elements of the set $\mc A$ on the complex plane.
The lines $y=0$, $y=\sqrt{3}x$ and $y=-\sqrt{3}x$
divide $\mathbb{C}\backslash\{0\}$ into six half-open, half-closed sectors; among these, one must contain at least
$(\ab{\mc A}-1)/6$ elements of $\mc A$.
Let $\mc A_0$ be the subset of $\mc A$ containing these elements. Thus,
\[
\mc A_0\subset\mc A,\ \ \ |\mc A_0|\ge \dfrac{|\mc A|-1}{6}.
\]
Figure \ref{fig01} shows the partition into six parts.

\begin{figure}[h]
    \centering
    \includegraphics[scale=0.5]{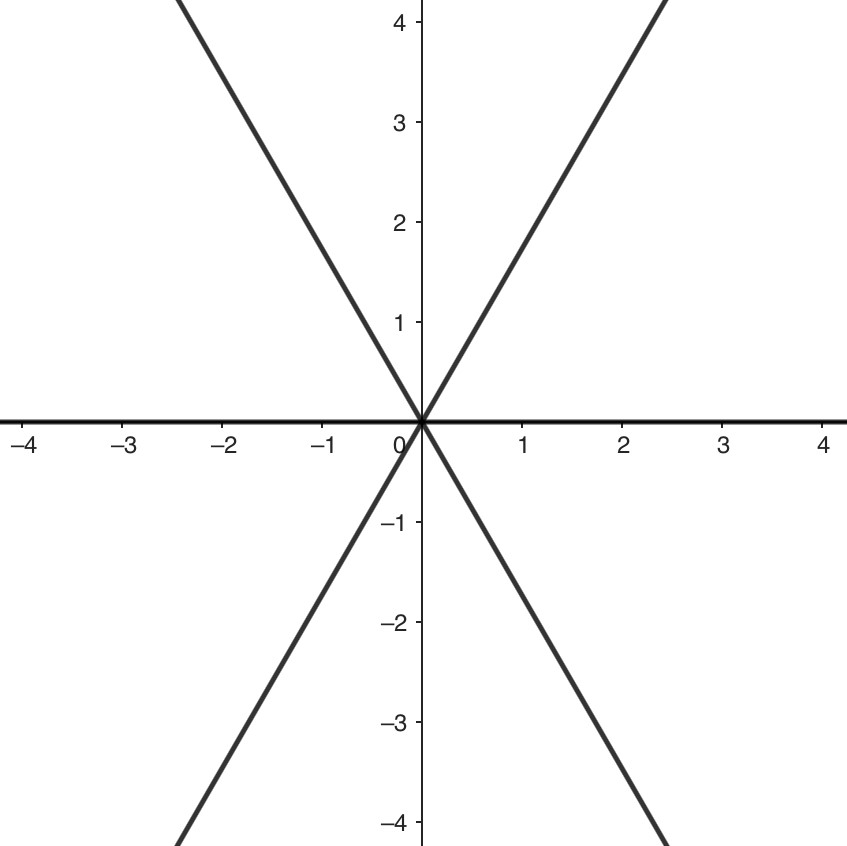}
    \caption{The partition of $\mathbb{C}\backslash\{0\}$ into six parts}
    \label{fig01}
\end{figure}

For $x\in\mathbb C$, denote the absolute value of $x$ by $|x|$.
Then, the norm of an Eulerian integer $a\in E$ is $N(a)=|a|^2$.
Since the angle between vectors of any two Eulerian integers in $\mc A_0$ is an acute angle, by the law of sines for $a,b\in\mc A_0$
\begin{align}
|a+b| >\max\{|a|,|b|\}.\lb{har01}
\end{align}

Since $\mc A_0\subseteq\mc A$,
\[
\omega_E(\prod_{\substack{a,b\in\mc A\\a\ne b}} (a+b))
\ge \omega_E(\prod_{\substack{a,b\in\mc A_0\\a\ne b}} (a+b))
\]
Let us list the odd-norm Euler prime divisors of $\prod_{\substack{a,b\in\mc A\\ a\ne b}} (a+b)$ by choosing exactly one of the prime divisors associated with each other (arbitrarily).

Let these be
$\pi_1,\pi_2,\pi_3,\dots,\pi_s$.
We prove that
\begin{align}
s> \dfrac{\log (|\mc A|-1)-\log 18}{\log 2}.
\lb{s01}
\end{align}
This implies Theorem \ref{thm01}.
We then prove \eqref{s01} by contradiction.
Suppose that
\begin{align}
s\leq\dfrac{\log (|\mc A|-1)-\log 18}{\log 2}.
\lb{s02}
\end{align}

We recursively define a sequence of sets
\[
\mc A_0\supseteq\mc A_1\supseteq\mc A_2\supseteq \dots \supseteq \mc A_s
\]
such that
\begin{align}
|\mc A_{i+1}|\ge \dfrac{|\mc A_i|}{2}
\lb{large}
\end{align}
if $0\le i\le s-1$, and
\begin{align}
a,b\in\mc A_{i+1},\ \alpha\in\mathbb N,
\pi_{i+1}^{\alpha}\mid a+b \ \ \ \Leftrightarrow
\ \ \pi_{i+1}^{\alpha}\mid a,\ \textup{ and } \pi_{i+1}^{\alpha}\mid b.
\lb{kov1}
\end{align}
To do this, we first divide the mod $\pi_{i+1}$ reduced residue classes
into two groups, $U_i$ and $V_i$, such that
a reduced residue class $n$ and its additive inverse $-n$
never fall into the same group
(since $2\nmid{\pi_{i+1}}$ they are different reduced residue classes).
We then partition the set $\mc A_i$ into two subsets
$\mc A_{i,0}$ and $\mc A_{i,1}$, by writing each element $a\in\mc A_i$ in the form
$a=\pi_{i+1}^\gamma a_0$, where $\pi_{i+1} \nmid a_0$.
Let $a\in\mc A_{i,0}$ if $a_0\in U_i$ and
$a\in\mc A_{i,1}$ if $a_0\in V_i$.
It is clear that, for $j\in\{0,1\}$, we have
\begin{align}
a,b\in\mc A_{i,j},\ \alpha\in\mathbb N,
\pi_{i+1}^{\alpha}\mid a+b \ \ \ \Leftrightarrow
\ \ \ \pi_{i+1}^{\alpha}\mid a,\ \textup{ and } \pi_{i+1}^{\alpha}\mid b,
\lb{kov2}
\end{align}
since if $a=\pi_{i+1}^{\gamma}a_0,\ b=\pi_{i+1}^{\delta}b_0$
(where $\pi_{i+1}\nmid a_0,b_0$) we consider two cases.

\bigskip
\noindent\textbf{Case 1:}
$\gamma\ne \delta$. Then $\pi_{i+1}^{\min\{\gamma,\delta\}}\mid a+b$,
but $\pi_{i+1}^{\min\{\gamma,\delta\}+1}\nmid a+b$. 
This implies \eqref{kov1}.

\bigskip
\noindent\textbf{Case 2:}
$\gamma=\delta$. Then
\[
a+b=\pi_{i+1}^{\gamma} (a_0+b_0),
\]
where $a_0,b_0\in U_i$ or $a_0,b_0\in V_i$. Due to the definition of the sets $U_i$ and
$V_i$, it is impossible that
$a_0\not\equiv -b_0\mod{\pi_{i+1}}$, i.e.
$\pi_{i+1}\nmid a_0+b_0$. That is, $\pi_{i+1}^{\gamma}\mid a+b$, but
$\pi_{i+1}^{\gamma+1}\nmid a+b$. Since
$\pi_{i+1}^{\gamma}\mid a,b$, this also verifies
\eqref{kov1}.

Next, we turn to the recursive definition of the sets
$\mc A_0\supseteq\mc A_1\supseteq\mc A_2\supseteq \dots \supseteq \mc A_s$
. If the sets
$\mc A_0\supseteq\mc A_1\supseteq\mc A_2\supseteq \dots \supseteq \mc A_i$
are already constructed with the desired property, then let
$\mc A_{i+1}$ simply be the set with the larger number of elements between $\mc A_{i,0}$ and $\mc A_{i,1}$. If the numbers of elements are equal,
then we can arbitrarily choose
which subset should be $\mc A_{i+1}$.
Then both \eqref{large} and \eqref{kov1} are satisfied.

For all odd-norm Euler primes $\pi_1,\pi_2,\dots,\pi_s$
we have that
\begin{align}
a,b\in\mc A_s,\ \alpha\in\mathbb N,
\pi_{i}^{\alpha}\mid a+b \ \ \ \Leftrightarrow
\ \ \ \pi_{i}^{\alpha}\mid a,\ \textup{ and } \pi_{i}^{\alpha}\mid b,
\lb{kov3}
\end{align}
and, by \eqref{s02} and \eqref{large},
\[
|\mc A_s|\ge \dfrac{|\mc A_0|}{2^s}
\ge\dfrac{|\mc A|-1}{6\cdot 2^s}
\ge 3.
\]
Let $a,b,c$ be three distinct elements of $\mc A_s$.
Let us write the prime factorization of $a+b,a+c,b+c$ in the ring of
Eulerian integers:
\begin{align*}
a+b &= \varepsilon_1 2^{\gamma_0}
\pi_1^{\gamma_1}\pi_2^{\gamma_2}\dots \pi_s^{\gamma_s},\\
a+c &= \varepsilon_2 2^{\beta_0}
\pi_1^{\beta_1}\pi_2^{\beta_2}\dots \pi_s^{\beta_s},\\
b+c &= \varepsilon_3 2^{\alpha_0}
\pi_1^{\alpha_1}\pi_2^{\alpha_2}\dots \pi_s^{\alpha_s}.
\end{align*}

It is clear that, by \eqref{kov3},
\[
\pi_1^{\gamma_1}\pi_2^{\gamma_2}\dots \pi_s^{\gamma_s}\mid a,b.
\]
Furthermore, we must have
\[
2^{\gamma_0}\nmid a, b.
\]
To see this, suppose, for example, that $2^{\gamma_0}\mid a$.
Then, since we already established that
$\pi_1^{\gamma_1}\pi_2^{\gamma_2}\dots \pi_s^{\gamma_s}\mid a$, we would have
$2^{\gamma_0}\pi_1^{\gamma_1}\pi_2^{\gamma_2}\dots \pi_s^{\gamma_s}\mid a$,
and consequently
\[
|a+b|=|2^{\gamma_0}\pi_1^{\gamma_1}\pi_2^{\gamma_2}\dots \pi_s^{\gamma_s}|
\le |a|,
\]
which contradicts \eqref{har01}. Similarly, we can show that
$2^{\gamma_0}\nmid b$.
Then, the exponent of $2$ in the prime factorization of $a$ and $b$ must be the same. Otherwise
$2^{\gamma_0} \nmid a+b$.
Consequently, the exponent of $2$ is identical in $a, b,$ and $c$.
Let
\begin{align*}
a=2^ta_1,\\
b=2^tb_1,\\
c=2^tc_1,
\end{align*}
where the norms of the Eulerian integers $a_1,b_1,c_1$ are odd.
Substituting these into the factorizations above, we get:
\begin{align*}
a_1+b_1&= \varepsilon_1 2^{\gamma_0-t}
\pi_1^{\gamma_1}\pi_2^{\gamma_2}\dots \pi_s^{\gamma_s}\\
a_1+c_1 &= \varepsilon_2 2^{\beta_0-t}
\pi_1^{\beta_1}\pi_2^{\beta_2}\dots \pi_s^{\beta_s}\\
b_1+c_1 &= \varepsilon_3 2^{\alpha_0-t}
\pi_1^{\alpha_1}\pi_2^{\alpha_2}\dots \pi_s^{\alpha_s}.
\end{align*}
where $\pi_1^{\gamma_1}\pi_2^{\gamma_2}\dots \pi_s^{\gamma_s}\mid a_1,b_1$.
So
\begin{align*}
|a_1|,|b_1| \ge |\pi_1^{\gamma_1}\pi_2^{\gamma_2}\dots \pi_s^{\gamma_s}|.
\end{align*}

We prove that $\gamma_0-t\geq2$. We consider two cases.

\bigskip\noindent
\textbf{Case 1:}
$|a_1|\neq|b_1|$. By symmetry,
we may assume that $|a_1|>|b_1|$. So,
\begin{align*}
|a_1| &> |\pi_1^{\gamma_1}\pi_2^{\gamma_2}\dots \pi_s^{\gamma_s}|\\
|b_1| &\ge |\pi_1^{\gamma_1}\pi_2^{\gamma_2}\dots \pi_s^{\gamma_s}|.
\end{align*}
If the norm of $x\in E$ Eulerian integer is denoted by $N(x)=|x|^2$, then
\[
N(a_1)>N(\pi_1^{\gamma_1}\pi_2^{\gamma_2}\dots \pi_s^{\gamma_s}),
\]
but $\pi_1^{\gamma_1}\pi_2^{\gamma_2}\dots \pi_s^{\gamma_s}\mid a_1$, so
$N(\pi_1^{\gamma_1}\pi_2^{\gamma_2}\dots \pi_s^{\gamma_s})\mid N(a_1)$,
i.e.,
\begin{align*}
N(a_1) &\ge 2N(\pi_1^{\gamma_1}\pi_2^{\gamma_2}\dots \pi_s^{\gamma_s})\\
|a_1| &\ge \sqrt{2} |\pi_1^{\gamma_1}\pi_2^{\gamma_2}\dots \pi_s^{\gamma_s}|.
\end{align*}
Since $a,b\in\mc A_0$, the angle between $a$ and $b$ is $\gamma <60^{\circ}$,
and thus, the angle between $a_1$ and $b_1$ is also $\gamma < 60^{\circ}$.
By the law of cosines,
\begin{align*}
|a_1+b_1|^2 &= |a_1|^2+|b_1|^2+2\cos\gamma |a_1|\cdot |b_1|\\
&\ge |a_1|^2+|b_1|^2+2\cos 60^{\circ} |a_1|\cdot |b_1|\\
&= |a_1|^2+|b_1|^2+ |a_1|\cdot |b_1|\\
&\ge (1+\sqrt{2}^2+1\cdot\sqrt{2})
|\pi_1^{\gamma_1}\pi_2^{\gamma_2}\dots \pi_s^{\gamma_s}|^2\\
&> 4 |\pi_1^{\gamma_1}\pi_2^{\gamma_2}\dots \pi_s^{\gamma_s}|^2,
\end{align*}
which, in the first case, gives $\gamma_0-t\ge 2$.

\bigskip\noindent
\textbf{Case 2:}
$|a_1|=|b_1|$. In this case,
the angle $\gamma$ between $a_1$ and $b_1$ falls in the interval $(0^\circ,60^\circ)$, the law of cosines yields
 $$\sqrt3|a_1|<|a_1+b_1|<2|a_1|.$$
Now $|a_1+b_1|=2^{\gamma_0-t}|\pi_1^{\gamma_1}\pi_2^{\gamma_2}\dots \pi_s^{\gamma_s}|$, which leads to a contradiction in the case of $|a_1|=|b_1|=|\pi_1^{\gamma_1}\pi_2^{\gamma_2}\dots \pi_s^{\gamma_s}|$. Thus $|\pi_1^{\gamma_1}\pi_2^{\gamma_2}\dots \pi_s^{\gamma_s}|^2\mid|a_1|^2$ implies $$|a_1|\geq\sqrt2|\pi_1^{\gamma_1}\pi_2^{\gamma_2}\dots \pi_s^{\gamma_s}|.$$
Then $|a_1+b_1|\geq\sqrt3|a_1|\geq\sqrt6|\pi_1^{\gamma_1}\pi_2^{\gamma_2}\dots \pi_s^{\gamma_s}|$; thus, $$|a_1+b_1|^2=2^{2(\gamma_0-t)}|\pi_1^{\gamma_1}\pi_2^{\gamma_2}\dots \pi_s^{\gamma_s}|^2>2^2|\pi_1^{\gamma_1}\pi_2^{\gamma_2}\dots \pi_s^{\gamma_s}|^2$$
and thus $\gamma_0-t\geq2$. Similarly, $\beta_0-t\ge 2$ and
$\alpha_0-t\ge 2$. So, $a_1+b_1$, $a_1+c_1$ and $b_1+c_1$
are numbers divisible by $2^2=4$ (in the ring $E$), i.e.,
\begin{align*}
a_1+b_1 &= 4x,\\
a_1+c_1 &= 4y,\\
b_1+c_1 &= 4z,
\end{align*}
where $x,y,z$ are Eulerian integers. This shows that $a_1=2(x+y-z)$ is divisible by $2$, i.e., $a_1$ is twice an Eulerian integer. This contradicts the assumption that the norm of $a_1$ is odd.
Since our initial assumption \eqref{s02} leads to a contradiction,
the original statement of Theorem \ref{thm01} follows.

\bigskip
\noindent\textbf{Proof of Theorem \ref{thm02}.}
\bigskip In the case $\rho=-1$, if for $\pi$ an Euler prime such that $|\mathcal{A}|\geq{|\pi|^2+1}$, then there are two distinct Eulerian integers in $\mathcal{A}$ whose difference is divisible by $\pi$. Thus, $\omega_E(\prod_{\substack{a,b\in{\mathcal{A}}\\a\neq{b}}}(a-b))$ is at least the number of pairwise not associated Euler primes whose absolute value not greater than $\sqrt{|\mathcal{A}|-1}$, which is at least the number of rational
primes not greater than $\sqrt{|\mathcal{A}|-1}$. This can be estimated from below for sufficiently large sets $\mathcal{A}$ by the positive constant multiple of $\frac{\sqrt{|\mathcal{A}|-1}}{\log(\sqrt{|\mathcal{A}|-1})}$, from which the theorem can be stated for this case as well.
We can assume that
\begin{align*}
\rho \ne 1\ \ \ \ \textup{and}\ \ \ \ \rho\ne -1
\end{align*}
the theorem has already been proven for $\rho=1$
(see Theorem \ref{thm01}).

\bigskip
Every Euler prime has an associated $\pi$ such that
$\textup{Arg}(\pi)\in [0^\circ,60^\circ)$. Henceforth, we will
focus only on these associates, and let $\mc P$ be the set of
all such
Euler primes. Thus,
\[
\mc P=\{\pi:\ \pi\textup{ is an Euler prime and } \textup{Arg}(\pi)
\in [0^\circ,60^\circ) \}.
\]
The following lemma is the cornerstone of the proof of the theorem.

\begin{lem}\lb{lem01}
Let $\rho_0\in E$ and $\pi$ be an Euler prime
such that $\pi\nmid \rho_0$,
$\delta\in\mathbb N$, where $\pi^{\delta}\mid 1+\rho_0$,
but $\pi^{\delta+1}\nmid 1+\rho_0$.
Then, the reduced residue classes mod $\pi^{\delta+1}$ can be
partitioned into three disjoint groups,
$C_1, C_2, C_3$, such that for $1\le i\le 3$,
\[
a,b\in C_i\ \ \ \Rightarrow \ \ \ \pi^{\delta+1}\nmid a+\rho_0b.
\]
\end{lem}

\bigskip
\noindent\textbf{Proof of Lemma \ref{lem01}.}
Let us list the reduced residue classes mod $\pi^{\delta+1}$:
$r_1,r_2,\dots,r_m$.
These can be divided into three groups using a greedy algorithm such that
\[
r_i,\ -\rho_0 r_i
\]
never fall into the same group. Indeed,
assume that the elements
$r_1,r_2,\dots,r_{i-1}$ have been correctly assigned.
When assigning $r_i$, we must ensure that $r_i$ falls into
a different group than both the group of
\begin{align}
-\rho_0 r_i\ \ \ \ \textup{and}\ \ \ -\rho_0^{-1}r_i. \lb{csop01}
\end{align}
We have three groups, $C_1,C_2,C_3$.
Since the elements in \eqref{csop01} exclude at most two groups
(i.e., the groups to which $-\rho_0r_i$ and $-\rho_0r_i^{-1}$
are assigned),
a conflict can only arise if $r_i$ itself is an excluded element:
\begin{align*}
r_i &\equiv -\rho_0r_i \pmod{\pi^{\delta+1}}
\intertext{or (which is equivalent anyway)}
r_i &\equiv -\rho_0^{-1} r_i \pmod{\pi^{\delta+1}}
\end{align*}
(since in the case of $r_i \equiv -\rho_0r_i\pmod{\pi^{\delta+1}}$
$r_i$ and $-\rho_0r_i$
are definitely in the same group, given that they are identical).

But this case cannot occur since
\[
r_i \equiv -\rho_0r_i \pmod{\pi^{\delta+1}}
\]
implies
\[
\pi^{\delta+1} \mid (1+\rho_0)r_i,
\]
where $r_i$ and $\pi$ are relatively primes, so
\[
\pi^{\delta+1} \mid 1+\rho_0,
\]
which contradicts the conditions of the lemma.

\bigskip
\hfill \break
We now divide the rest of the proof of Theorem \ref{thm02} into two cases based on whether $\rho$ is the negative of an Euler prime power in $\mathcal{P}$ with a positive integer exponent.
(Recall that $\mc P$ contains
Euler primes $\pi$ such that $\textup{Arg}(\pi)\in[0^\circ, 60^\circ)$).

\hfill \break \textbf{Case 1:} Assume that $\rho$ is not the negative of a power of an Euler prime
$\theta\in\mc P$, i.e., $\rho$ is not of the form $\rho=-\theta^\gamma$ with $\theta\in\mc P$, $\gamma\in\mathbb N$.

Before the next lemma, we introduce a new notation.
Let $\rho\in E$ and $\pi$ be an Euler prime,
such that $\textup{Arg}(\pi)\in[0^\circ,60^\circ)$.
Write $\rho$ in the form
\[
\rho=\pi^{\gamma}\rho_0,
\]
where $\pi\nmid \rho_0$ (here, $\rho_0\neq-1$). Then, write $1+\rho_0$ in the form
\[
1+\rho_0=\pi^{\delta}\sigma,
\]
where $\pi\nmid \sigma$. Let $c(\pi, \rho)$ denote the
\[
c(\pi,\rho)\stackrel{\textup{def}}{=}\gamma+\delta \ge 0
\]
integer. It is clear that if
\[
\pi\nmid \rho (1+\rho)
\]
then $c(\pi,\rho)=0$. Thus,
\[
c(\rho) \stackrel{\textup{def}}{=}
\prod_{\substack{\pi\mid \rho(1+\rho)\\\textup{Arg}(\pi)\in[0^\circ,60^\circ)}} \pi^{c(\pi,\rho)},
\]
the value of which solely depends on $\rho$. Let $\tau(c(\rho))$ denote the
number of 
distinct divisors of $c(\rho)$ 
in the ring of Eulerian integers $E$, where two associated divisors
$\delta$ and $\varepsilon\delta$ 
are considered distinct
if $\varepsilon\ne 1$.

Then, Theorem \ref{thm02} will be proved with the constant
\[
c=\frac{\log(\tau(c(\rho))^2+2)}{\log 3}
.\]
 Namely, we will prove
that if $|\mathcal{A}|=3^s(\tau(c(\rho))^2+2)$ for an $s\in\mathbb N$, then,
in the first case,
$\omega_E(\prod_{\substack{a,b\in{\mathcal{A}}\\a\neq{b}}}(a+\rho{b}))>s$.

\begin{lem}\lb{lem02}
Let $\rho\in E$, $\mc A\subset E$ be a finite set and $\pi$ an
Euler prime such that $\textup{Arg}(\pi)\in[0^\circ,60^\circ)$ and
$-\rho$ is not a power of $\pi$ with a positive integer exponent. Then there exists a set $\mc B\subset\mc A$ such that
\[
|\mc B| \ge \dfrac{|\mc A|}{3},
\]
and if $a,b\in\mathcal{B}$, for every integer $u\ge c(\pi,\rho)$,
\begin{align}
\pi^{u}\mid a+\rho b\ \ \ \Rightarrow \ \ \
\pi^{u-c(\pi,\rho)}\mid a,\ b.\lb{fou}
\end{align}
\end{lem}

\bigskip
\noindent\textbf{Proof of Lemma \ref{lem02}.} Let
\[
\rho=\pi^{\gamma}\rho_0,
\]
where $\pi\nmid\rho_0$, and
\[
1+\rho_0=\pi^{\delta}\sigma,
\]
where $\pi\nmid\sigma$. Then, $1+\rho_0\neq0$. Let us write all elements of the set $\mc A$ in the form
\[
a=\pi^{\alpha}a_0,
\]
where $\pi\nmid a_0$, and let $C_1,C_2,C_3$ be the partition of the reduced residue classes $\mod \pi^{\delta+1}$ according to Lemma \ref{lem01}. Let
\begin{align*}
\mc B_1\stackrel{\textup{def}}{=}
\{a\in\mc A:\ a=\pi^{\gamma}a_0,\ \pi\nmid a_0,\ a_0\in C_1\}\\
\mc B_2\stackrel{\textup{def}}{=}
\{a\in\mc A:\ a=\pi^{\gamma}a_0,\ \pi\nmid a_0,\ a_0\in C_2\}\\
\mc B_3\stackrel{\textup{def}}{=}
\{a\in\mc A:\ a=\pi^{\gamma}a_0,\ \pi\nmid a_0,\ a_0\in C_3\}.
\end{align*}
Let $\mc B$ be the set with the most elements among
 $\mc B_1,\ \mc B_2,\ \mc B_3$
(if there are several sets with the same number of elements,
we choose $\mc B$ to
be any of them).
Obviously
\[
|\mc B| \ge \dfrac{|\mc A|}{3}.
\]
It remains is to prove that \eqref{fou} holds for $\mc B$.
For this, we write
\[
a=\pi^{\alpha}a_0\ \ \ \ \textup{and}\ \ \ b=\pi^{\beta}b_0
\]
in the form where $\pi\nmid a_0,b_0$. Then, $a_0$ and $b_0$ are elements of the same
$C_i$ set, say $a_0,b_0\in C_1$. We then distinguish two cases.

\bigskip\noindent Case I: $\alpha\ne \beta+\gamma$. Then,
\begin{align*}
\pi^{\min\{\alpha,\beta+\gamma\}}
\mid a+\rho b\ (=\pi^{\alpha}a_0+\pi^{\gamma}\rho_0\pi^{\beta}b_0),
\end{align*}
but
\begin{align*}
\pi^{\min\{\alpha,\beta+\gamma\}+1}
\nmid a+\rho b\ (=\pi^{\alpha}a_0+\pi^{\beta+\gamma}\rho_0b_0).
\end{align*}
We also have
\begin{align*}
\pi^{\alpha}&\mid a\ \ \ \ \textup{and}\ \ \ \ \pi^{\beta}\mid b\\
\pi^{\min\{\alpha,\beta+\gamma\}-\gamma}&\mid a,b,\\
\pi^{\min\{\alpha,\beta+\gamma\}-c(\pi,\rho)}&\mid a,b,
\end{align*}
and thus the lemma is proved in this case.

\bigskip\noindent Case II: $\alpha=\beta+\gamma$. Then,
\[
a+\rho b=\pi^{\alpha}a_0+\pi^{\beta+\gamma}\rho_0b_0
=\pi^{\beta+\gamma} (a_0+\rho_0b_0).
\]
By Lemma \ref{lem01} $\pi^{\delta+1}\nmid a_0+\rho_0b_0$,
so
\begin{align}
\pi^{\beta+\gamma}&\mid a+\rho b,\lb{eeg1}\\
\intertext{but}
\pi^{\beta+\gamma+\delta+1}&\nmid a+\rho b,\notag\\
\pi^{\beta+c(\pi,\rho)+1}&\nmid a+\rho b. \lb{eeg2}
\end{align}
Since
\begin{align}
\pi^{\beta}\mid \pi^{\beta+\gamma}\ \ \textup{and} \ \ \pi^{\beta+\gamma}=
\pi^{\alpha}\mid a, \ \ \
\textup{ we have } \ \ \ \pi^{\beta}\mid a,\ b.\lb{eeg3}
\end{align}
Thus, the lemma follows from \eqref{eeg1}, \eqref{eeg2}, and \eqref{eeg3} in this case
completing the proof of the lemma.

\bigskip
Let us list the Euler prime divisors of the product $\prod_{\substack{a,b\in\mc A\\ a\ne b}}(a+\rho{b})$
with arguments in the interval $[0^\circ,60^\circ)$ (of which every Euler prime divisor has exactly one associate) $$\pi_1,\pi_2,\dots,\pi_s.$$ Suppose that
\[
s\le \dfrac{\log (|\mc A|/(\tau(c(\rho))^2+2))}{\log 3}
= \dfrac{\log (|\mc A|)-\log(\tau(c(\rho))^2+2)}{\log 3},
\]
from which we aim to obtain a contradiction.
We then recursively define
\[
\mc A_{0}\stackrel{\textup{def}}{=}
\mc A\supset \mc A_1\supset \mc A_2
\supset\dots\supset \mc A_s
\]
of sets such that
\[
|\mc A_{i}| \ge \dfrac{|\mc A_{i-1}|}{3},
\]
and if $u\ge c(\pi_i,\rho)$,
\begin{align*}
a,b\in\mc A_i,\ \pi_{i}^{u}\mid a+\rho b\ \ \ \Rightarrow \ \ \
\pi_i^{u-c(\pi_i,\rho)}\mid a,\ b.
\end{align*}
This is easy to do, given
$\mc A_0,\mc A_1,\dots,\mc A_{i-1}$.
We then apply Lemma \ref{lem02}, with the choice $\mc A=\mc A_{i-1}$ and $\pi=\pi_i$,
and the resulting set $\mc B$ gives $\mc A_i$. Then,
\[
|\mc A_s| \ge \dfrac{ |\mc A_{s-1}|}{3}
\ge \dfrac{|\mc A_{s-2}|}{3^2}\ge\dots
\ge \dfrac{|\mc A_{0}|}{3^s}=\dfrac{|\mc A|}{3^s}\ge \tau(c(\rho))^2+2,
\]
and for each of the primes $\pi_1,\pi_2,\dots,\pi_s$, we have that
if $u\ge c(\pi_i,\rho)$,
\begin{align}
a,b\in\mc A_s,\ \pi_{i}^{u}\mid a+\rho b\ \ \ \Rightarrow \ \ \
\pi_i^{u-c(\pi_i,\rho)}\mid a,\ b.\lb{focsi1}
\end{align}
In the following, we denote $\Phi(a,b)$ to be the Eulerian integer defined as:
\[
\Phi(a,b)=\dfrac{a}{\textup{gcd}(a,b)}
+\rho\dfrac{b}{\textup{gcd(a,b)}}.
\]
If $$a=\varepsilon_1\pi_1^{\alpha_1}\pi_2^{\alpha_2}\dots\pi_s^{\alpha_s}$$ and $$b=\varepsilon_2\pi_1^{\beta_1}\pi_2^{\beta_2}\dots\pi_s^{\beta_s},$$ where $\varepsilon_1^6=1=\varepsilon_2^6$, $\alpha_1,\alpha_2,\dots,\alpha_s,\beta_1,\beta_2,\dots,\beta_s\in\mathbb{N}\cup\{0\}$ and $\pi_1,\pi_2\dots,\pi_s\in\mc{P}$, then $$\gcd(a,b)=\pi_1^{\min\{\alpha_1,\beta_1\}}\pi_2^{\min\{\alpha_2,\beta_2\}}\dots\pi_s^{\min\{\alpha_s,\beta_s\}}.$$
We prove the following:
\begin{lem}\lb{lem03}
\[
|\{\Phi(a,b):\ a,b\in \mc{A}_s\}| \le \tau(c(\rho)).
\]
\end{lem}

\bigskip\noindent\textbf{ Proof of Lemma \ref{lem03}.}
For $a,b\in \mc A_s$, write $a+\rho b$ in the form
\begin{align*}
&a+\rho b =\varepsilon\pi_1^{\alpha_1}\cdots\pi_s^{\alpha_s}\\
&=\underbrace{\zj{\varepsilon\pi_1^{\min\{c(\pi_1,\rho),\alpha_1\}}
\cdots \pi_s^{\min\{c(\pi_s,\rho),\alpha_s\}}}}_{u}
\cdot
\underbrace{\zj{\pi_1^{\alpha_1-\min\{c(\pi_1,\rho),\alpha_1\}}
\cdots \pi_s^{\alpha_s-\min\{c(\pi_s,\rho),\alpha_s\}}}}_{v}\\
&\stackrel{\textup{def}}{=} uv,
\end{align*} where $\varepsilon^6=1$. If $\alpha_i\ge c(\pi_i,\rho)$, then, according to \eqref{focsi1},
\begin{align*}
\pi_i^{\alpha_i-c(\pi_i,\rho)}\mid a,b,\\
\pi_i^{\alpha_i-\min\{c(\pi_i,\rho),\alpha_i\}}\mid a,b.
\end{align*}

If $\alpha_i< c(\pi_i,\rho)$, then
\[
\pi_i^{\alpha_i-\min\{c(\pi_i,\rho),\alpha_i\}}=1\mid a,b.
\]
Since this holds for all $1\le i\le s$;
thus, $v\mid a$ and $v\mid b$,
so $v\mid\textup{gcd}(a,b)$.
It follows that
\[
\dfrac{a}{\textup{gcd}(a,b)}+\rho
\dfrac{b}{\textup{gcd}(a,b)}\mid \dfrac{a}{v}+\rho \dfrac{b}{v}=u.
\]
Furthermore,
\[
u\mid \pi_1^{c(\pi_1,\rho)}\cdots\pi_s^{c(\pi_s,\rho)}\mid{c(\rho)} .
\]
Thus, $\Phi(a,b)\mid c(\rho)$, which completes the proof of the lemma.

\bigskip
From the lemma, it is clear that
\begin{align}
|\{(\Phi(a,b),\Phi(b,a)):\ a,b\in \mathcal{A}_s\}| \le \tau(c(\rho))^2.
\lb{tau1}
\end{align}
Let $t=\tau(c(\rho))^2$.
Since $|\mc A_s|\ge \tau(c(\rho))^2+2=t+2$,
we can select $t+2$ distinct elements from $\mc A_s$.
Let $a$ denote the first of these elements, and let $b_1,b_2,\dots,b_{t+1}$ denote the others.
By \eqref{tau1},
\[
|\{(\Phi(a,b_i),\Phi(b_i,a)):1\le s\le t+1\}|
\le \tau(c(\rho))^2=t.
\]
By the pigeonhole principle, there exist different $i$ and $j$ such that
\[
(\Phi(a,b_i),\Phi(b_i,a))=(\Phi(a,b_j),\Phi(b_j,a)).
\]
Let
\[
\Phi(a,b_i)=\Phi(a,b_j)=z_1
\]
and
\[
\Phi(b_i,a)=\Phi(b_j,a)=z_2.
\]
The pairs
$\dfrac{a}{\textup{gcd} (a,b_i)},\ \dfrac{b_i}{\textup{gcd} (a,b_i)}$
and
$\dfrac{a}{\textup{gcd} (a,b_j)},\ \dfrac{b_j}{\textup{gcd} (a,b_j)}$
are the solutions of the same system of two linear equations
(with two unknowns)
\begin{align*}
x+\rho y=z_1\\
\rho x+y=z_2.
\end{align*}
Since
$\textup{det}\begin{bmatrix}1& \rho \\ \rho&1\end{bmatrix}\ne 0$,
the above system of equations has an unique solution, i.e.
\begin{align}
\dfrac{a}{\textup{gcd}(a,b_i)} = \dfrac{a}{\textup{gcd}(a,b_j)}\lb{ww11}\\
\dfrac{b_i}{\textup{gcd}(a,b_i)} = \dfrac{b_j}{\textup{gcd}(a,b_j)}\lb{ww21}
\end{align}
From \eqref{ww11}, it follows that
$\textup{gcd }(a,b_i)
=\textup{gcd }(a,b_j)$, and by substituting this into \eqref{ww21} we obtain,
$b_i=b_j$, which is a contradiction. This proves the theorem in Case 1.

\bigskip\noindent\textbf{Case 2:} There is an Euler prime $\theta\in\mc P$ with argument $[0^\circ,60^\circ)$ and a positive integer $\gamma$
such that $\rho=-\theta^\gamma$.

\bigskip\noindent Lemma 2 of the proof of the previous part and the notation $c(\pi,\rho)$ for Euler primes with argument $[0^\circ,60^\circ)$
that are not $\theta$,
can still be applied, in which case $\rho=\rho_0$
and $1+\rho_0=\pi^\delta\sigma$.
Consequently, if $\mc A\subset E$ is a finite set and $\pi$ is an Euler prime with argument in the interval $[0^\circ,60^\circ)$,
which is different from $\theta$,
there exists a set $\mc B\subset\mc A$ for which
$|\mc B| \ge \dfrac{|\mc A|}{3}$,
such that for every integer $u\ge c(\pi,\rho)$ and $a,b\in\mathcal{B}$,
$\pi^{u}\mid a+\rho b\ \ \ \Rightarrow \ \ \
\pi^{u-c(\pi,\rho)}\mid a,\ b$. The proof proceeds similarly. Finally, we can divide the set by $\theta$ as follows:

\begin{lem}\lb{lem04}
Let $\theta\in\mc P$ be an Euler prime with argument in the interval $[0^\circ,60^\circ)$ and let $\gamma$ be a positive integer such that
$\rho=-\theta^\gamma$.
Then, if $\mathcal{A}\subset{E}$ is a finite set, there exists a
subset $\mathcal{B}\subset\mathcal{A}$ such that $|\mathcal{B}|\geq\frac{|\mathcal{A}|}{2}$, and if $a,b\in\mathcal{B}$, $a\neq{b}$, then $\theta^k\mid{a+\rho{b}=a-\theta^\gamma{b}}$ implies $\theta^{k-\gamma}\mid{a},b$.
\end{lem}
\bigskip\noindent\textbf{Proof of Lemma \ref{lem04}.}
We begin by writing
the elements of the set $\mathcal{A}$ in the form $\theta^{\alpha}a_0$, where $\theta\nmid{a_0}$, $0\leq\alpha\in\mathbb{Z}$ and $a_0\in{E}$.
We then partition $\mc A$ into two subsets:
$$\mathcal{A}_1=\{\theta^{\alpha}a_0=a\in\mathcal{A}:\ \theta\nmid a_0,\ \Big[\frac{\alpha}{2\gamma}\Big]\equiv0\pmod2\}$$ and $$\mathcal{A}_2=\{\theta^{\alpha}a_0=a\in\mathcal{A}:\ \theta\nmid b_0,\ \Big[\frac{\alpha}{2\gamma}\Big]\equiv1\pmod2\}.$$ Here, $0$ can be arbitrarily partitioned into either of the two sets. Let $\mathcal{B}$
be the larger of $\mathcal{A}_1$ and $\mathcal{A}_2$
(or either one if they have the same cardinality).
Consider any two distinct elements $a,b\in\mc B$.
If these elements are written in the form $a=\theta^{\alpha_a}a_0$ and
$b=\theta^{\alpha_b}b_0$, the absolute difference of their exponents,
$|\alpha_a-\alpha_b|$, cannot be $\gamma$.
Thus, in the case of $a,b\in\mathcal{B},\ (a\neq{b})$, the highest exponent of $\theta$ that divides the numbers $a$ and
$\rho{b}\ (=-\theta^{\gamma}b)$ must be different. Thus, in the case of $\theta^k\mid{a+\rho{b}=a-\theta^\gamma{b}}$, $\theta^{k}\mid{a},\rho{b}$, so $\theta^{k-\gamma}\mid{a},b$.
This completes the proof of Lemma \ref{lem04}.

\bigskip

Now we consider the case when $-\rho=\theta^{\gamma}$,
where $0<\gamma\in\mathbb Z$ and $\theta\in\mc P$. Let
\[
c(\theta,\rho)\stackrel{\textup{def}}{=}\gamma,
\]
and let
$$c(\rho)\stackrel{\textup{def}}{=}\prod_{\substack{\pi\mid{\rho(1+\rho)}\\\textup{Arg}(\pi)\in[0^\circ,60^\circ)}}\pi^{c(\pi,\rho)}.$$
(Here the product runs on the prime divisors of $\rho(1+\rho)$, where the prime divisors
in this case also include $\theta$.)

\bigskip
Then the method used in Case 1 will also be applicable:
Let us list the Euler prime divisors of $\prod_{\substack{a,b\in\mc A\\ a\ne b}}(a+\rho{b})$
whose arguments lie in the interval
$[0^\circ,60^\circ)$: $\pi_1,\pi_2,\dots,\pi_s$.
Suppose that
\[
s\le \dfrac{\log (|\mc A|/(\tau(c(\rho))^2+2))}{\log 3}
= \dfrac{\log (|\mc A|)-\log(\tau(c(\rho))^2+2)}{\log 3},
\]
Our aim is to derive a contradiction from this assumption.

We then recursively define a sequence of sets
\[
\mc A_{0}\stackrel{\textup{def}}{=}
\mc A\supset \mc A_1\supset \mc A_2
\supset\dots\supset \mc A_s
\]
for which
\[
|\mc A_{i}| \ge \dfrac{|\mc A_{i-1}|}{3},
\]
and if $u\ge c(\pi_i,\rho)$,
\begin{align*}
a,b\in\mc A_i,\ \pi_{i}^{u}\mid a+\rho b\ \ \ \Rightarrow \ \ \
\pi_i^{u-c(\pi_i,\rho)}\mid a,\ b.
\end{align*}

This is easy to do, since, given
$\mc A_0,\mc A_1,\dots,\mc A_{i-1}$,
we can apply Lemma \ref{lem02} or Lemma \ref{lem04}. $\mc A=\mc A_{i-1}$ and $\pi=\pi_i$. We use Lemma 4 exactly when $-\rho$ is a power of $\pi$ with a positive integer exponent, and the resulting set $\mc B$ gives $\mc A_i$. Instead of the halving obtained in Lemma 4, we can also third the set there. Then,
\[
|\mc A_s| \ge \dfrac{ |\mc A_{s-1}|}{3}
\ge \dfrac{|\mc A_{s-2}|}{3^2}\ge\dots
\ge \dfrac{|\mc A_{0}|}{3^s}=\dfrac{|\mc A|}{3^s}\ge \tau(c(\rho))^2+2,
\]
and for each of the primes $\pi_1,\pi_2,\dots,\pi_s$, the following holds:
if $u\ge c(\pi_i,\rho)$,
\begin{align}
a,b\in\mc A_s,\ \pi_{i}^{u}\mid a+\rho b\ \ \ \Rightarrow \ \ \
\pi_i^{u-c(\pi_i,\rho)}\mid a,\ b.\lb{focsi}
\end{align}

In the following, let $\Phi(a,b)$ be the Eulerian integer defined as:
\[
\Phi(a,b)=\dfrac{a}{\textup{gcd}(a,b)}
+\rho\dfrac{b}{\textup{gcd(a,b)}}.
\]
Similarly to Lemma 3, \[
|\{\Phi(a,b):\ a,b\in \mathcal{A}_s\}| \le \tau(c(\rho))
\] now holds with the same reasoning.

\bigskip
From this, it is clear that
\begin{align}
|\{(\Phi(a,b),\Phi(b,a)):\ a,b\in \mathcal{A}_s\}| \le \tau(c(\rho))^2.
\lb{tau}
\end{align}
Then, let $t=\tau(c(\rho))^2$.
Since $|\mc A_s|\ge \tau(c(\rho))^2+2=t+2$, we can select $t+2$ distinct elements from $\mc A_s$.
We denote the first element by $a$, and the others by
$b_1,b_2,\dots,b_{t+1}$.
Then, by \eqref{tau},
\[
|\{(\Phi(a,b_i),\Phi(b_i,a)):1\le s\le t+1\}|
\le \tau(c(\rho))^2=t.
\]
Thus, by the pigeonhole principle, there exist distinct $i$ and $j$ such that
\[
(\Phi(a,b_i),\Phi(b_i,a))=(\Phi(a,b_j),\Phi(b_j,a)).
\]
Let
\[
\Phi(a,b_i)=\Phi(a,b_j)=z_1
\]
and
\[
\Phi(b_i,a)=\Phi(b_j,a)=z_2.
\]
Then,
$\dfrac{a}{\textup{gcd} (a,b_i)},\ \dfrac{b_i}{\textup{gcd} (a,b_i)}$
and
$\dfrac{a}{\textup{gcd} (a,b_j)},\ \dfrac{b_j}{\textup{gcd} (a,b_j)}$
are the solutions of the same system of linear equations with two unknowns
\begin{align*}
x+\rho y=z_1\\
\rho x+y=z_2.
\end{align*}
 Since
$\textup{det}\begin{bmatrix}1& \rho \\ \rho&1\end{bmatrix}\ne 0$,
the above system of linear
equations has an unique solution, i.e.,
\begin{align}
\dfrac{a}{\textup{gcd}(a,b_i)} = \dfrac{a}{\textup{gcd}(a,b_j)}\lb{ww1}\\
\dfrac{b_i}{\textup{gcd}(a,b_i)} = \dfrac{b_j}{\textup{gcd}(a,b_j)}\lb{ww2}
\end{align}
Then, by \eqref{ww1}, $\textup{gcd}(a,b_i)
=\textup{gcd}(a,b_j)$, but then, by \eqref{ww2},
$b_i=b_j$, which is a contradiction. Our result now follows.

\bigskip\noindent\textbf{Proof of Corollary \ref{cor01}. }
Since $\omega = \frac{-1+\sqrt{3}i}{2}$ is a root of unity in the ring of
Eulerian integers $E$, we use the notation corresponding to the proof of Case 1 of Theorem 2: $c(\omega)=1$ and $\tau(c(\omega))=6$. Thus, the lower bound on the number of distinct Eulerian prime factors ($\omega_E$) follows directly from the proof of Theorem 2:
$$\omega_E\left(\prod_{\substack{a,b\in{\mathcal{A}}\\a\neq{b}}}(a+\omega{b})\right)>\frac{\log|\mathcal{A}|-\log(\tau(c(\omega))^2+2)}{\log 3}=\frac{\log|\mathcal{A}|-\log38}{\log 3}.$$

We connect this result to the rational prime factors by observing that the product $\prod (a^2-ab+b^2)$ is the norm of the product $\prod (a+\omega b)$:
$$\prod_{\substack{a,b\in{\mathcal{A}}\\a\neq{b}}}(a^2-ab+b^2) = N\left(\prod_{\substack{a,b\in{\mathcal{A}}\\a\neq{b}}}(a+\omega{b})\right) = \prod_{\substack{a,b\in{\mathcal{A}}\\a\neq{b}}}(a+\omega{b})(a+\omega^2b).$$

For any rational prime $p$, the product $p$ factors in $E$ into at most two distinct conjugate Eulerian primes. Thus, the number of distinct rational prime factors, $\omega_{\mathbb{N}}$, of the product is at least half the number of distinct Eulerian prime factors $\omega_E$.

Therefore, this implies:
$$\omega_{\mathbb{N}}\left(\prod_{\substack{a,b\in{\mathcal{A}}\\a\neq{b}}}(a^2-ab+b^2)\right)>\frac{1}{2}\left(\frac{\log|\mathcal{A}|-\log38}{\log 3}\right),$$
which is the required result.

\bigskip\noindent\textbf{Proof of Corollary \ref{cor02}. }
This proof is similar to the proof of Corollary 1. Since $-\omega$ is an Eulerian integer that is a root of unity and $-\omega+1=\frac{3-\sqrt{3}i}{2}=\sqrt{3}i(\frac{-1-\sqrt{3}i}{2})$ is an Euler prime, with the notation corresponding to the proof of Case 1 of Theorem 2, $c(-\omega)=-\sqrt{3}i\omega$ and $\tau(c(-\omega))=\tau(-\sqrt{3}i\omega)=12$.

Thus, the lower bound on the number of distinct Euler prime factors ($\omega_E$) follows directly from the proof of Theorem 2 (with $\rho=-\omega$):
$$\omega_E\left(\prod_{\substack{a,b\in{\mathcal{A}}\\a\neq{b}}}(a-\omega{b})\right)>\frac{\log|\mathcal{A}|-\log(\tau(c(-\omega))^2+2)}{\log 3}=\frac{\log|\mathcal{A}|-\log146}{\log 3}.$$

We relate this result to the rational prime factors by noting that the product $\prod (a^2+ab+b^2)$ is the norm of the product $\prod (a-\omega b)$:
$$\prod_{\substack{a,b\in{\mathcal{A}}\\a\neq{b}}}(a^2+ab+b^2) = N\left(\prod_{\substack{a,b\in{\mathcal{A}}\\a\neq{b}}}(a-\omega{b})\right) = \prod_{\substack{a,b\in{\mathcal{A}}\\a\neq{b}}}(a-\omega{b})(a-\omega^2b).$$

The number of distinct rational prime factors, $\omega_{\mathbb{N}}$, of the rational product $\prod (a^2+ab+b^2)$ is related to the number of distinct Euler prime factors $\omega_E$. Since any rational prime $p$ is the product of at most two distinct conjugate Euler prime factors, we have $\omega_{\mathbb{N}} \ge \omega_E / 2$. This implies:
$$\omega_{\mathbb{N}}\left(\prod_{\substack{a,b\in{\mathcal{A}}\\a\neq{b}}}(a^2+ab+b^2)\right)>\frac{1}{2}\left(\frac{\log|\mathcal{A}|-\log146}{\log 3}\right),$$
which yields the required result.

\bigskip\noindent\textbf{Proof of Theorem \ref{thm03}.}
\bigskip Let $n \geq 2$ be an integer. We define the sets $\mathcal{A}'$ and $\mathcal{B}'$ in $\mathbb{Z}^n$ for the elements of the sets $\mathcal{A}$ and $\mathcal{B}$:
\begin{align*}
\mathcal{A}'&=\{(r_1x^{m_1},r_2x^{m_2},\dots,r_{n-1}x^{m_{n-1}},1): x\in\mathcal{A}\}\\
\mathcal{B}&'=\{(1,y,y^2,\dots,y^{n-2},r_ny^{n-1}): y\in\mathcal{B}\}
\end{align*}
Since $|\mathcal{A}| \geq |\mathcal{B}|$, we have $|\mathcal{A}'|\geq|\mathcal{B}'|\geq{2n-2}$.

\bigskip  The following theorem is included in the paper of Győry, Sárközy, and Stewart \cite{GyoryStewartSarkozy}:

\bigskip\noindent\textbf{Theorem D.} [Győry-Sárközy-Stewart]
\textit{Let $n\geq2$ be an integer, and let $\mathcal{A},\mathcal{B}\subset(\mathbb{Z}^+)^n$ be finite sets such that $|\mathcal{A}|\geq|\mathcal{B}|\geq{2n-2}$. If for every vector in $\mathcal{A}$ the $n$-th coordinate is $1$, and any $n$ vectors in $\mathcal{B} \cup \{(0,\dots,0,1)\}$ are linearly independent, then there exists an effectively computable positive constant $c$ for which
$$\omega_{\mathbb{N}}\left(\prod_{\substack{(a_1,\dots,a_n)\in\mathcal{A}\\(b_1,\dots,b_n)\in\mathcal{B}}}(a_1b_1+a_2b_2+\dots+a_nb_n)\right)>c\log|\mathcal{A}|.$$}
\break

\bigskip\noindent
To ensure we can apply Theorem D, we must verify that its conditions hold for our sets $\mc A'$ and $\mc B'$.
\begin{enumerate}
    \item The $n$-th coordinate of every vector in $\mathcal{A}'$ is $1$.
    \item We show that any $n$ distinct vectors from $\mathcal{B}' \cup \{(0,\dots,0,1)\}$ are linearly independent.
\end{enumerate}

Condition 1 is immediate. We proceed to prove Condition 2.

\bigskip\noindent\textbf{Case 1: }All $n$ vectors are from $\mathcal{B}'$.
Let the vectors correspond to $y_1, y_2, \dots, y_n \in \mathcal{B}$. The determinant of the matrix formed by these vectors is proportional to a Vandermonde determinant:
$$
\begin{vmatrix}
1 & 1 & \dots & 1 \\
y_1 & y_2 & \dots & y_{n} \\
\dots & \dots & \dots & \dots \\
y_1^{n-2} & y_2^{n-2} & \dots & y_{n}^{n-2} \\
r_ny_1^{n-1} & r_ny_2^{n-1} & \dots & r_ny_{n}^{n-1} \\
\end{vmatrix} = r_n \cdot \prod_{1\le i<j\le n} (y_j-y_i) \neq 0.
$$
Since the elements $y_i$ are distinct and $r_n \neq 0$, the determinant is non-zero, proving linear independence.

\bigskip\noindent\textbf{Case 2:} One vector is $(0,\dots,0,1)$.
Let the $n-1$ vectors correspond to $y_1, y_2, \dots, y_{n-1} \in \mathcal{B}$, and the $n$-th vector be $\mathbf{e}_n = (0,\dots,0,1)$. The determinant of the matrix formed by these vectors is found by cofactor expansion along the last column:
$$
\begin{vmatrix}
1 & 1 & \dots & 1 & 0 \\
y_1 & y_2 & \dots & y_{n-1} & 0 \\
\dots & \dots & \dots & \dots & \dots \\
y_1^{n-2} & y_2^{n-2} & \dots & y_{n-1}^{n-2} & 0 \\
r_ny_1^{n-1} & r_ny_2^{n-1} & \dots & r_ny_{n-1}^{n-1} & 1 \\
\end{vmatrix} = 1 \cdot \begin{vmatrix}
1 & 1 & \dots & 1 \\
y_1 & y_2 & \dots & y_{n-1} \\
\dots & \dots & \dots & \dots \\
y_1^{n-2} & y_2^{n-2} & \dots & y_{n-1}^{n-2} \\
\end{vmatrix} \neq 0.
$$
The resulting subdeterminant is a non-zero Vandermonde determinant since $y_i$ are distinct. Thus, the vectors are linearly independent.

\bigskip
Since the assumptions are met, we apply Theorem D to the sets $\mathcal{A}'$ and $\mathcal{B}'$:
$$\omega_{\mathbb{N}}\left(\prod_{\substack{\mathbf{a}\in\mathcal{A}'\\ \mathbf{b}\in\mathcal{B}'}}\mathbf{a}\cdot\mathbf{b}\right)>c\log|\mathcal{A}'|=c\log|\mathcal{A}|,$$
where the inner product $\mathbf{a}\cdot\mathbf{b}$ yields:
\begin{align*}
\mathbf{a}\cdot\mathbf{b} &= (r_1x^{m_1})(1) + (r_2x^{m_2})(y) + \dots + (r_{n-1}x^{m_{n-1}})(y^{n-2}) + (1)(r_ny^{n-1}) \\
&= r_1x^{m_1}+r_2x^{m_2}y+\dots+r_{n-1}x^{m_{n-1}}y^{n-2}+r_ny^{n-1} = f(x,y).
\end{align*}
Thus,
$$\omega_{\mathbb{N}}\left(\prod_{x\in\mathcal{A},y\in\mathcal{B}}f(x,y)\right) > c\log|\mathcal{A}|$$
for an effectively computable positive constant $c$, as required.

\end{document}